\documentclass[11pt]{amsart}

\usepackage{amsmath,amssymb,amsthm,mathtools,comment,bigints}
\usepackage[hyphens]{url}
\usepackage[usenames,dvipsnames]{color}
\usepackage[pagebackref,linktocpage=true,colorlinks=true,linkcolor=RoyalBlue,citecolor=BrickRed,urlcolor=RoyalBlue]{hyperref}
\usepackage[msc-links,abbrev]{amsrefs}

\newcommand{\R}{\mathbf R}
\newcommand{\Sph}{\mathbf S}
\newcommand{\C}{\mathcal C}
\newcommand{\dist}{\operatorname{dist}}
\newcommand{\diam}{\operatorname{diam}}
\newcommand{\inrad}{\operatorname{inrad}}
\newcommand{\width}{\operatorname{width}}
\newcommand{\interior}{\operatorname{int}}
\newcommand{\WF}{\operatorname{WF}}

\newcommand{\Radon}{\mathcal R}

\newtheorem{theorem}{Theorem}[section]
\newtheorem{lemma}[theorem]{Lemma}
\newtheorem{proposition}[theorem]{Proposition}

\theoremstyle{definition}
\newtheorem{note}[theorem]{Note}

\begin{document}

\title[A Converse to the Archimedes Sphere Theorem]
{A Converse to the Archimedes Sphere Theorem}

\author{Mohammad Ghomi}
\address{School of Mathematics, Georgia Institute of Technology, Atlanta, GA 30332}
\email{ghomi@math.gatech.edu}
\urladdr{ghomi.math.gatech.edu}

\author{Ralph Howard}
\address{Department of Mathematics, University of South Carolina, Columbia, SC 29208}
\email{howard@math.sc.edu}
\urladdr{people.math.sc.edu/howard}

\author{Mijia Lai}
\address{School of Mathematical Sciences, Shanghai Jiao Tong University, Shanghai 200240}
\email{laimijia@sjtu.edu.cn}
\urladdr{math.sjtu.edu.cn/faculty/mlai}

\begin{abstract}
We show that a convex surface  in Euclidean $3$-space is a
sphere if the area cut off from it by a slab of some fixed
width depends only on the direction of the slab. This result
generalizes theorems of Blaschke and Stamm, which assumed the
slab-area property for every width. It holds also for any smooth closed embedded surface, and we obtain analogous results in higher dimensions. The proof is based on an
averaging identity, which relates slab areas to the Newtonian
potential of the surface, the electrostatic characterization of
spheres by Mendez and Reichel, and analytic microlocal
properties of the Radon transform, which are used to show that the
slab-area property forces real analyticity, via theorems of Kashiwara and Boman.
\end{abstract}

\date{Last revised on \today}

\maketitle

\section{Introduction}\label{sec:intro}

Archimedes showed that the area cut off from a sphere by a slab depends only on the width of the slab
\cite{archimedes2004}.  Blaschke
\cite{blaschke1921}  proved a converse: a smooth strictly convex surface whose
slab areas depend only on the width of the slab
is a sphere. Stamm \cites{stamm1951} extended this result by allowing the areas to depend also on the direction of the slabs. A problem
posed in \cite{ghomi2017} asks
whether it suffices to consider a single
width. We show that the answer is yes.

A \emph{convex surface} $\Gamma$ is the boundary of a \emph{convex body}, i.e.,
a compact convex set $K\subset\R^3$ with
nonempty interior. A \emph{slab} is the closed region between two
parallel planes, and its width is the distance between them.
For a compact surface $\Gamma$, we say that $\Gamma$ has the
\emph{slab-area property} for $h>0$ if, for each direction
$u\in\Sph^2$, all slabs of width $h$ orthogonal to $u$ whose
boundary planes both intersect $\Gamma$ cut off the same area from
$\Gamma$. Let $\width(\Gamma)$ denote the minimum distance between
two parallel supporting planes of $\Gamma$.

\begin{theorem}\label{thm:archimedes}
If $\Gamma\subset\R^3$ is a convex surface with the slab-area property for some
$0<h<\width(\Gamma)$, then $\Gamma$ is a sphere.
\end{theorem}

The convexity assumption may be omitted when there is sufficient regularity:

\begin{theorem}\label{thm:smooth}
Let $\Gamma\subset\R^3$ be a $\C^3$ closed embedded surface.
If $\Gamma$ has the slab-area property for some
$0<h<\width(\Gamma)$, then $\Gamma$ is a sphere.
\end{theorem}

The proof of Theorem~\ref{thm:archimedes} begins with an averaging identity
(Section~\ref{sec:potential}): if $x\in K$ lies at
distance  $> h/2$ from $\Gamma$, the average area of the
width-$h$ slabs centered at $x$ is $2\pi h$ times the Newtonian
potential $U(x)$ of $\Gamma$. When $h$ is less than twice the
inradius of $\Gamma$, this forces $U$ to be constant on $\Gamma$,
and the electrostatic characterization of spheres by Mendez and
Reichel \cite{mendezreichel2000} quickly establishes Theorem \ref{thm:archimedes}.

For general
$h$, we first treat the real analytic case
(Section~\ref{sec:analytic}) by considering
the area density of $\Gamma$ with respect to height in each
direction. The slab-area property makes these densities periodic. Real analyticity allows them to be continued to a bounded
periodic density. Averaging this continuation over
directions produces a bounded real analytic function on $\R^3$ which
agrees with $U$ on the interior of $K$. Since $U$ is harmonic there,
analyticity implies that this function is harmonic on $\R^3$,
and hence constant by Liouville's theorem. 

To complete the proof, it remains to show that the slab-area property
forces real analyticity. We first prove, by elementary means, that
it forces $\C^1$ regularity (Section~\ref{sec:general}). Next we
show that a $\C^1$ surface satisfying the condition is real
analytic (Section~\ref{sec:C1}), using
analytic microlocal properties of the Radon transform, Kashiwara's
watermelon theorem \cite{kashiwara1970,sjostrand1994},
and a recent regularity result of Boman \cite{boman2022regularity}.

To prove Theorem~\ref{thm:smooth}, we show that the periodicity of
the area densities forces the mean curvature of $\Gamma$ to be
nonpositive inside its convex hull, which by the maximum principle
makes $\Gamma$ convex (Section~\ref{sec:smooth}). Theorem~\ref{thm:archimedes}
then applies.

Finally, in Section~\ref{sec:higher-dimensions} we extend the case
$h<2\inrad(\Gamma)$ of Theorems~\ref{thm:archimedes}
and~\ref{thm:smooth} to hypersurfaces in $\R^n$, $n\ge3$, where
slabs are replaced by cylinders with $2$-dimensional axes
(Theorem~\ref{thm:hd}). This is a single-radius version of the
``Archimedean projection property'' of the sphere $\Sph^{n-1}$ studied by Coll,
Dodd, and Harrison \cite{coll-dodd-harrison2017}.
See
\cites{rudin1973,richmond-richmond1993,dodd-coll2008,kim-kim2012,
yaskin-yaskina2015,jeronimo-yee2019,knothe1957}
for other studies related to the Archimedean property of the
sphere. Theorems~\ref{thm:archimedes} and \ref{thm:smooth}
are the first known results for a single width or
radius.

\begin{note}
If the common area in the slab-area property is the same in all directions, as in  the original question \cite{ghomi2017},
then Theorems~\ref{thm:archimedes} and \ref{thm:smooth} hold under the weaker assumption
$h<\diam(\Gamma)$. Indeed, if
$\width(\Gamma)\le h<\diam(\Gamma)$, continuity yields a direction
in which the width is $h$, so the common area is the total area of
$\Gamma$. But in any direction of width greater than $h$, a slab of
width $h$ omits a cap of $\Gamma$, a contradiction.
\end{note}

\section{Newtonian Potential}\label{sec:potential}
Throughout Sections~\ref{sec:potential}--\ref{sec:C1}, $K\subset\R^3$
is a convex body and $\Gamma\coloneqq \partial K$, unless stated otherwise.
For $x\in K$, let
$$
U(x)\coloneqq \frac1{4\pi}\int_{y\in\Gamma}\frac1{|x-y|}
$$
be the \emph{Newtonian potential} of $\Gamma$ at $x$. The proof of Theorem \ref{thm:archimedes} is based on the following characterization of the sphere.

\begin{lemma}[Mendez--Reichel \cite{mendezreichel2000}]
\label{lem:mendez-reichel}
If $U$
is constant on $\Gamma$, then $\Gamma$ is a sphere.
\end{lemma}

This result had been conjectured by Gruber \cite[p.~353]{heil-martini1993} and proved by
Reichel \cite{reichel1997} for smooth surfaces.
We also need the following fact which is proved by a standard argument, cf. \cite[p.~59]{landkof1972}. Let $\interior(K)$ denote the interior of $K$.

\begin{lemma}\label{lem:continuity}
$U$ is continuous on $K$ and harmonic on $\interior(K)$.
\end{lemma}

\begin{proof}
Harmonicity follows by differentiating under the
integral sign.
To prove continuity, let $B_\rho(x)$ be the ball of radius $\rho$.
By monotonicity of surface area for nested convex bodies,
$f(t)\coloneqq |\Gamma\cap B_t(x)|\le|\partial B_t(x)|=4\pi t^2$.
The Lebesgue--Stieltjes measure $df$ is the push-forward of area
measure on $\Gamma$ by the map $y\mapsto|x-y|$. Hence
$$
\int_{y\in\Gamma\cap B_\rho(x)}\frac1{|x-y|}
=\int_0^\rho\frac{df(t)}{t}.
$$
Since $f(t)=O(t^2)$ as $t\to0^+$, integration by parts gives
$$
\int_0^\rho\frac{df(t)}{t}
=\frac{f(\rho)}{\rho}+\int_0^\rho\frac{f(t)}{t^2}\,dt
\le8\pi\rho.
$$
Now let $U_\rho$ be defined as $U$, but with $|x-y|^{-1}$ replaced
by $\min\{\rho^{-1},|x-y|^{-1}\}$. Then $U_\rho$ is continuous on
$K$, since its integrand is bounded and continuous, and
$0\le U-U_\rho\le2\rho$ on $K$ by the last estimate. Thus $U$ is
continuous, being a uniform limit of continuous functions.
\end{proof}

For $h>0$, $x\in\R^3$, and $u\in\Sph^2$, let
$
S_h(x,u)\coloneqq 
\{y\in\R^3:|\langle y-x,u\rangle|\le h/2\}
$
be the \emph{slab} of width $h$ with midplane
through $x$ and normal $u$.
For $h>0$, set
$
K_h\coloneqq \{x\in K:\dist(x,\Gamma)>h/2\}.
$
We write $|E|$ for the area of a set $E\subset\Gamma$.

\begin{lemma}[The averaging identity]\label{lem:average}
For every $x\in K_h$,
$$
\frac1{4\pi}\int_{u\in\Sph^2}
\bigl|\Gamma\cap S_h(x,u)\bigr|
=2\pi hU(x).
$$
\end{lemma}

\begin{proof}
Fix $x\in K_h$. For $y\in\Gamma$, let
$
Z(y)\coloneqq \{u\in\Sph^2:|\langle y-x,u\rangle|\le h/2\}
$
be the zone of $\Sph^2$ between the two planes
orthogonal to $y-x$ at distance $h/(2|x-y|)$ from the origin.
Since $|x-y|>h/2$, this distance is less than $1$, so both
planes meet $\Sph^2$, and the zone has width $h/|x-y|$.
By the theorem of Archimedes, 
$$
|Z(y)|=\frac{2\pi h}{|x-y|}.
$$
Next note that $y\in S_h(x,u)$ if and only if $u\in Z(y)$,
since both conditions state that $|\langle y-x,u\rangle|\le h/2$.
Hence, for each $u\in\Sph^2$,
$$
\bigl|\Gamma\cap S_h(x,u)\bigr|
=\int_{y\in\Gamma}\mathbf 1_{S_h(x,u)}(y)
=\int_{y\in\Gamma}\mathbf 1_{Z(y)}(u).
$$
Integrating over $u$ and applying Fubini's theorem, we obtain
\begin{align*}
\int_{u\in\Sph^2}\bigl|\Gamma\cap S_h(x,u)\bigr|
&=
\int_{u\in\Sph^2}\int_{y\in\Gamma}
\mathbf 1_{Z(y)}(u)\\
&=\int_{y\in\Gamma}\int_{u\in\Sph^2}
\mathbf 1_{Z(y)}(u)=\int_{y\in\Gamma}\bigl|Z(y)\bigr|
=8\pi^2 hU(x),
\end{align*}
which completes the proof.
\end{proof}

The above results quickly yield Theorem~\ref{thm:archimedes} for widths which are not too large.
From now on, through Section~\ref{sec:C1}, we assume that $\Gamma$
has the slab-area property for some fixed $0<h<\width(\Gamma)$.
The \emph{inradius} $\inrad(\Gamma)$ of $\Gamma$ is the radius of a
largest ball contained in $K$.

\begin{proposition}\label{prop:inradius}
If $h<2\inrad(\Gamma)$, then $\Gamma$ is a sphere.
\end{proposition}

\begin{proof}
Since $K_h$ is nonempty, for each fixed $u\in\Sph^2$ the quantity
$|\Gamma\cap S_h(x,u)|$ is independent of $x\in K_h$. Hence its average
over $u$ is independent of $x$. By Lemma~\ref{lem:average}, $U$ is
constant on $K_h$. Since $U$ is harmonic on the connected set $\interior(K)$ and is
constant on the nonempty open set $K_h$, unique continuation shows
that $U$ is constant on $\interior(K)$. By
Lemma~\ref{lem:continuity}, it has the same constant value on $\Gamma$,
and Lemma~\ref{lem:mendez-reichel} completes the proof.
\end{proof}

\section{The Real Analytic Case}\label{sec:analytic}
Here we show that  Theorem~\ref{thm:archimedes} holds when $\Gamma$ is real analytic.
We begin by recording some facts which hold for any convex
surface.
For $u\in\Sph^2$, let
$
\phi_u(y)\coloneqq \langle y,u\rangle
$
be the \emph{height function} in the direction $u$, on $\R^3$ and,
by restriction, on $\Gamma$. The \emph{support function} of
$\Gamma$ is given by
$
p(u)\coloneqq \max_{\Gamma}\phi_u.
$
For $-p(-u)<t<p(u)$, let
$$
\Gamma_{u,t}\coloneqq \{y\in\Gamma:\phi_u(y)=t\}
$$
be the level sets of $\phi_u$ on $\Gamma$.
Since $\Gamma$ is locally Lipschitz, it has a unique outward unit
normal $\nu$ almost everywhere, and there the tangential gradient
of $\phi_u$ is
$
\nabla\phi_u(y)=u-\langle\nu(y),u\rangle\nu(y).
$
Thus
\begin{equation}\label{eq:R}
R(u,t)\coloneqq 
\int_{y\in\Gamma_{u,t}}
\frac1{|\nabla\phi_u(y)|}
=
\int_{y\in\Gamma_{u,t}}\frac{1}{\sqrt{1-\langle\nu(y),u\rangle^2}}
\end{equation}
is defined for almost every $t$, and the coarea formula gives
\begin{equation}\label{eq:coarea}
\bigl|\{y\in\Gamma:a\le\phi_u(y)\le b\}\bigr|
=
\int_a^b R(u,t)\,dt
\end{equation}
for $-p(-u)<a<b<p(u)$. We call $R(u,t)$ the \emph{area density}
in direction $u$ and set $R(u,t)=0$ for $t<-p(-u)$ or $t>p(u)$.
Let
$
D\coloneqq \{(u,t)\in\Sph^2\times\R:-p(-u)<t<p(u)\}.
$

\begin{lemma}\label{lem:R-analytic}
If $\Gamma$ is $\C^1$, $R$ is continuous on $D$. If $\Gamma$ is real analytic, 
$R$ is real analytic on $D$.
\end{lemma}

\begin{proof}
Fix $(u_0,t_0)\in D$, and choose a point $c_0\in\interior(K)$ with
$\phi_{u_0}(c_0)=t_0$. For $(u,t)$ near $(u_0,t_0)$, the plane
$\{\phi_u=t\}$ meets the fixed line $c_0+\R u_0$ at
$$
c(u,t)\coloneqq 
c_0+\frac{t-\langle c_0,u\rangle}{\langle u_0,u\rangle}\,u_0.
$$
Thus $c(u,t)$ depends real analytically on $(u,t)$ and, after shrinking
the neighborhood, belongs to $\interior(K)$.
Choose a real analytic orthonormal frame $e_1(u),e_2(u)$ of $u^\perp$
defined on a neighborhood of $u_0$ in $\Sph^2$, and put 
$
e(u,\theta)\coloneqq \cos\theta\,e_1(u)+\sin\theta\,e_2(u).
$
For every nearby $(u,t)$ and $\theta$, the ray
$c(u,t)+r e(u,\theta)$ meets $\Gamma$ in a unique point
$$
y(u,t,\theta)=c(u,t)+r(u,t,\theta)e(u,\theta).
$$
The intersection is transverse, since $c(u,t)\in\interior(K)$ and a
supporting plane at $y(u,t,\theta)$ separates $c(u,t)$ strictly from
that point. The implicit function theorem shows that $r$ is $\C^1$
locally, and real analytic if $\Gamma$ is real analytic. Uniqueness
makes the local solutions agree on overlaps, and compactness of the
parameter circle gives a common neighborhood of $(u_0,t_0)$.
Thus $y(u,t,\theta)$ is a $\C^1$ parametrization of $\Gamma_{u,t}$
depending $\C^1$ on $(u,t)$. Since $\partial_\theta y\ne0$,
$$
R(u,t)=
\bigintsss_0^{2\pi}
\frac{\big|\partial_\theta y(u,t,\theta)\big|}
{\sqrt{1-\big\langle\nu\big(y(u,t,\theta)\big),u\big\rangle^2}}\,d\theta
$$
is continuous. If $\Gamma$ is real analytic, then
$\partial_\theta y$ is real analytic and nonzero, so
$
|\partial_\theta y|
=
\sqrt{\langle\partial_\theta y,\partial_\theta y\rangle}
$
is real analytic. The denominator is likewise  real
analytic. Hence the integrand is real analytic. Compactness of the
parameter circle supplies uniform local analytic estimates in
$(u,t)$, so integration also gives real analyticity of $R$.
\end{proof}

Since $h<\width(\Gamma)\le p(u)+p(-u)$, for $-p(-u)<t<p(u)-h$
the planes at heights $t$ and $t+h$ both meet $\Gamma$, so by the
slab-area property and \eqref{eq:coarea}
$\int_t^{t+h}R(u,s)\,ds$ is independent of $t$. Hence, by
Lebesgue's differentiation theorem,
\begin{equation}\label{eq:periodic}
R(u,t+h)=R(u,t)
\end{equation}
for almost every $-p(-u)<t<p(u)-h$, and for every such $t$ when
$R$ is continuous.

\begin{proposition}\label{prop:analytic}
If $\Gamma$ is real analytic, then it is a sphere.
\end{proposition}
\begin{proof}
We extend $R$ periodically. For $(u,t)\in\Sph^2\times\R$, choose
$k\in\mathbf Z$ so that
$-p(-u)<t+kh<p(u)$, which is possible since
$p(u)+p(-u)>h$, and set
$$
\widetilde R(u,t)\coloneqq R(u,t+kh).
$$
This is well defined: if two choices of $k$ are possible, repeated use
of \eqref{eq:periodic} gives the same value. Since
$p$ is continuous, the same $k$ works in a neighborhood of
any fixed $(u,t)$; Lemma~\ref{lem:R-analytic} therefore shows that
$\widetilde R$ is real analytic. It is $h$-periodic in $t$, so it
descends to a real analytic function on the compact manifold
$\Sph^2\times(\R/h\mathbf Z)$ and in particular is bounded.
Define, for $x\in\R^3$,
$$
V(x)\coloneqq \frac1{8\pi^2}
\int_{u\in\Sph^2}\widetilde R(u,\phi_u(x)).
$$
The function $V$ is bounded and real analytic on $\R^3$. Indeed,
compactness of $\Sph^2\times(\R/h\mathbf Z)$ and real analyticity give
constants $C,L>0$ such that $|\partial_t^m\widetilde R(u,t)|\le CL^m m!$
for all $m\ge0$, and differentiation under the integral gives the
corresponding analytic estimates for $V$.

We claim that
$
V=U
$
on $\interior(K)$.
Fix $x\in\interior(K)$ and put $d\coloneqq \dist(x,\Gamma)$. If
$0<\varepsilon<2d$, then $x\in K_\varepsilon$, so
Lemma~\ref{lem:average}, with $h$ replaced by $\varepsilon$, and
\eqref{eq:coarea} give
$$
2\pi U(x)
=
\frac1{4\pi}\int_{u\in\Sph^2}
\frac1\varepsilon
\int_{\phi_u(x)-\varepsilon/2}^{
\phi_u(x)+\varepsilon/2}
R(u,t)\,dt.
$$
As $\varepsilon\to0^+$, the inner averages converge uniformly in $u$
to $R(u,\phi_u(x))$, since these parameters range over a
compact subset of $D$. Furthermore, since $(u,\phi_u(x))\in D$ for every $u\in\Sph^2$,
we have
$
\widetilde R(u,\phi_u(x))=R(u,\phi_u(x)).
$
Consequently,
$$
2\pi U(x)
=
\frac1{4\pi}\int_{u\in\Sph^2}R(u,\phi_u(x))
=
\frac1{4\pi}\int_{u\in\Sph^2}\widetilde R(u,\phi_u(x))
=
2\pi V(x),
$$
as desired. Now since $U$ is harmonic on $\interior(K)$, so is $V$. But $V$ is real
analytic on $\R^3$, hence $\Delta V$ is real analytic and vanishes on
the nonempty open set $\interior(K)$. Therefore $\Delta V=0$ on
$\R^3$. Thus $V$ is a bounded harmonic function on $\R^3$, and
Liouville's theorem yields that $V$ is constant. Hence $U$ is constant
on $\interior(K)$ and, by Lemma~\ref{lem:continuity}, on $\Gamma$.
Lemma~\ref{lem:mendez-reichel} completes the proof.
\end{proof}

\section{Basic Regularity}\label{sec:general}

Here we show that the slab-area property forces $\Gamma$ to be $\C^1$.
The proof is based on a uniform estimate
for the areas of \emph{caps} of $\Gamma$, i.e., portions cut off by a plane. 
First we record the following basic fact.

\begin{lemma}\label{lem:faces}
For any supporting plane $P$ of $\Gamma$, $|P\cap\Gamma|=0$.
\end{lemma}

\begin{proof}
Let $u$ be the outward unit normal of $P$, so that $P$ is at
height $p(u)$. For small $\varepsilon>0$, the slab between the
heights $p(u)-h$ and $p(u)$ and its translate between
$p(u)-h-\varepsilon$ and $p(u)-\varepsilon$ cut off the same area
from $\Gamma$. Removing their common
part, we obtain
$$
\bigl|\Gamma\cap\{p(u)-\varepsilon<\phi_u\le p(u)\}\bigr|
=\bigl|\Gamma\cap\{p(u)-h-\varepsilon\le\phi_u<p(u)-h\}\bigr|.
$$
As $\varepsilon\to0^+$, the sets on the right decrease to the empty
set, while those on the left decrease to $P\cap\Gamma$; hence
$|P\cap\Gamma|=0$.
\end{proof}

For $-p(-u)<t\le p(u)$, let
$
A(u,t)\coloneqq \bigl|\{y\in\Gamma:\phi_u(y)\ge t\}\bigr|
$
be the area of the cap of $\Gamma$ cut off by the plane
orthogonal to $u$ at height $t$, on the side of $u$. Thus
$A(u,p(u)-s)$ is the area of the cap of depth $s$ below the
supporting plane with outward normal $u$. By
Lemma~\ref{lem:faces}, \eqref{eq:coarea} also holds for $b=p(u)$;
in particular
\begin{equation}\label{eq:cap}
A(u,t)=\int_t^{p(u)}R(u,s)\,ds.
\end{equation}
Let $\width_u(\Gamma)\coloneqq p(u)+p(-u)$ denote the width of $\Gamma$ in the direction $u$.

\begin{lemma}\label{lem:cap-bounds}
There are constants $m,M>0$ such that
$
ms\le A(u,p(u)-s)\le Ms
$
for every $u\in\Sph^2$ and $0<s<\width_u(\Gamma)$.
\end{lemma}

\begin{proof}
Put $d\coloneqq \diam(\Gamma)$, $r\coloneqq \inrad(\Gamma)$, and
$\delta\coloneqq (\width(\Gamma)-h)/2>0$.
Let $c(u)\coloneqq (p(u)-p(-u))/2$ and $I_u\coloneqq [c(u)-h/2,c(u)+h/2]$, so that
$-p(-u)+\delta\le t\le p(u)-\delta$ for $t\in I_u$. At a point
$y\in\Gamma_{u,t}$ where there is a unique outward normal $\nu$, we claim 
\begin{equation}\label{eq:tilt}
|\nabla\phi_u(y)|\ge\delta/d.
\end{equation}
If $\langle\nu(y),u\rangle\ge0$, choose $z\in K$ with
$\phi_u(z)=p(u)$. Then $\langle z-y,\nu(y)\rangle\le0$, since $K$
lies below the supporting plane of $\Gamma$ at $y$, and so
$
\delta\le p(u)-t=\langle z-y,u\rangle
\le\langle z-y,\nabla\phi_u(y)\rangle
\le d\,|\nabla\phi_u(y)|,
$
as desired.
When $\langle\nu(y),u\rangle<0$, apply the same argument with $u$
replaced by $-u$. Thus \eqref{eq:tilt} holds.

We may assume that $B_r(0)\subset K$, after a translation. Then
$p(u)\le d$, and the planar section $K\cap\{\phi_u=t\}$
has diameter at most $d$. It also contains a disk of radius at least $r\delta/d$. 
This follows by taking the convex hull of $B_r(0)$ with a point
$z\in K$ satisfying $\phi_u(z)=p(u)$, or with the
corresponding point in direction $-u$ when $t<0$. Consequently, if $L(\Gamma_{u,t})$ denotes the length of
$\Gamma_{u,t}$, which is the boundary of the planar section, then
$2\pi r\delta/d\le L(\Gamma_{u,t})\le2\pi d$. On the other hand, by
\eqref{eq:tilt}, the integrand in \eqref{eq:R} lies between $1$
and $d/\delta$ on $\Gamma_{u,t}$. Hence
$$
m\coloneqq 2\pi r\delta/d\le L(\Gamma_{u,t})\le R(u,t)\le L(\Gamma_{u,t})d/\delta
\le2\pi d^2/\delta=:M
$$
for almost every $t\in I_u$. Since $I_u$ has length $h$, every
$t\in(-p(-u),p(u))$ differs from a point of $I_u$ by an integer
multiple of $h$. So by \eqref{eq:periodic} the same bounds hold for
almost every such $t$. Hence \eqref{eq:cap} yields the desired estimate.
\end{proof}

For the next two arguments, it is convenient to fix $q\in\Gamma$
and, for $u\in\R^3$ and $z\in K$, define the function
$$
f_u(z)\coloneqq \langle q-z,u\rangle.
$$
Then $f_u$ depends linearly on $u$, and a nonzero vector $u$ is a
supporting normal of $\Gamma$ at $q$ if and only if $f_u\ge0$ on $K$.

\begin{lemma}\label{lem:general-strict}
$\Gamma$ is strictly convex.
\end{lemma}
\begin{proof}
Suppose that $\Gamma$ contains a  segment
$q+[-\ell/2,\ell/2]e$, where $|e|=1$ and $\ell>0$.
Every supporting normal at $q$ is perpendicular to $e$. We show
that these normals form an open set in the circle
$S\coloneqq \Sph^2\cap e^\perp$.
Fix a supporting unit normal $n$ at $q$, so that $f_n\ge0$ on $K$,
and let $v\coloneqq n\times e$.
We claim that, for every $z\in K$,
\begin{equation}\label{eq:fvz}
|f_v(z)|\le(2M/\ell)f_n(z).
\end{equation}
Suppose first that $f_n(z)>0$, and choose $0<s<f_n(z)$. The cap body
$
K\cap\{0\le f_n(y)\le s\}
$
contains the triangle with vertices
$
q-\ell e/2,
$
$
q+\ell e/2,
$
and
$
q+(s/f_n(z))(z-q).
$
Its projection onto $n^\perp$ has area
$
\ell s|f_v(z)|/(2f_n(z)).
$
Every line parallel to $n$ which meets the cap body
$
K\cap\{0\le f_n\le s\}
$
also meets the surface cap
$
\Gamma\cap\{0\le f_n\le s\}.
$
Thus the orthogonal projection of the surface cap onto $n^\perp$
contains the projection of the triangle above. Since orthogonal
projection does not increase area, Lemma~\ref{lem:cap-bounds} yields
$$
Ms\ge A(n,p(n)-s)\ge\ell s|f_v(z)|/(2f_n(z)),
$$
as desired. If $f_n(z)=0$, then $z$ lies
in the supporting face of $K$ with normal $n$. If $f_v(z)\ne0$, that face would
contain, together with the given segment, a triangle of positive
area, contrary to Lemma~\ref{lem:faces}. Thus $f_v(z)=0$, and \eqref{eq:fvz} holds.
Now, if $|\varepsilon|<\ell/(2M)$, then for every $z\in K$,
$$
f_{n+\varepsilon v}(z)
=f_n(z)+\varepsilon f_v(z)
\ge
\left(1-2M|\varepsilon|/\ell\right)f_n(z)
\ge0.
$$
Hence $(n+\varepsilon v)/|n+\varepsilon v|$ is a supporting normal. So the set of supporting unit normals at $q$ is open in $S$. It is also nonempty and closed, and therefore equals $S$. This forces $K$ to lie in the line $q+\R e$, contrary
to $K$ having nonempty interior.
\end{proof}

\begin{proposition}\label{prop:regularity}
$\Gamma$ is $\C^1$.
\end{proposition}
\begin{proof}
It suffices to show that $\Gamma$ has a unique supporting plane at
each point \cite[Thm.~2.2.4]{schneider2014}. Suppose to the
contrary that $q\in\Gamma$ has two distinct outward unit normals
$n_1,n_2$. Set $n\coloneqq (n_1+n_2)/|n_1+n_2|$ and
$v\coloneqq (n_1-n_2)/|n_1-n_2|$. Then $n\perp v$, and
$n_1=c_1n+c_2v$, $n_2=c_1n-c_2v$ for some $c_1,c_2>0$. Since
$f_{n_1}=c_1f_n+c_2f_v$ and $f_{n_2}=c_1f_n-c_2f_v$ are nonnegative
on $K$, we have $|f_v|\le\alpha f_n$ on $K$, where
$\alpha\coloneqq c_1/c_2$. In particular $f_n\ge0$ on $K$, so $n$ is a
supporting normal at $q$. Consider the cap
$C_s\coloneqq K\cap\{f_n(z)\le s\}$, and let $e\coloneqq n\times v$. By Lemma~\ref{lem:general-strict},
the supporting face with normal $n$ consists only of $q$. Hence
there is a function $\varepsilon(s)\to0$ as $s\to0^+$ such that
$|f_e|\le\varepsilon(s)$ on $C_s$. Thus $C_s$ lies in a rectangular
box with side lengths $s$, $2\alpha s$, and $2\varepsilon(s)$. By
monotonicity of surface area for nested convex bodies,
$A(n,p(n)-s)\le|\partial C_s|\le4\alpha s^2+(8\alpha+4)s\varepsilon(s)$.
It follows that $A(n,p(n)-s)/s\to0$, contradicting the lower bound
in Lemma~\ref{lem:cap-bounds}.
\end{proof}

\section{Proof of Theorem \ref{thm:archimedes}}\label{sec:C1}

In view of Propositions \ref{prop:analytic} and \ref{prop:regularity}, to prove Theorem~\ref{thm:archimedes} it remains to show:

\begin{proposition}\label{prop:C1-analytic}
If $\Gamma$ is $\C^1$, then it is real analytic.
\end{proposition}

To prove the proposition we use the \emph{analytic wave front set}
$\WF_A(\mu)$ of a distribution $\mu$, which records the points and
directions in which $\mu$ fails to be real analytic; see
\cite[Chap.~8 \&~9]{hormander2003analysis}, \cite{sjostrand1982}.
Throughout this section $\mu$ denotes the area measure of $\Gamma$,
regarded as a compactly supported distribution on $\R^3$,
$\langle\mu,\psi\rangle\coloneqq \int_{y\in\Gamma}\psi(y)$, whose
support is $\Gamma$. We call a nonzero vector $\xi$ a
\emph{singular direction} of $\mu$ at $q$ if
$(q,\xi)\in\WF_A(\mu)$. We need the following facts:

\begin{lemma}[{\cite[Thm.~8.4.5]{hormander2003analysis}}]\label{lem:WF-basic}
Singular directions of $\mu$ occur only at points of $\Gamma$.
\end{lemma}

\begin{lemma}[{Kashiwara's Watermelon Theorem \cite[Cor.~9.6.8]{hormander2003analysis}}]
\label{lem:watermelon}
Let $q\in\Gamma$. If $\xi\notin\R\nu(q)$ is a singular direction
of $\mu$ at $q$, then so is every vector in the open half-plane of
$\operatorname{span}\{\nu(q),\xi\}$ bounded by $\R\nu(q)$ which
contains $\xi$.
\end{lemma}

\begin{lemma}[{Boman \cite[Thm.~1]{boman2022regularity}}]\label{lem:boman}
Let $q\in\Gamma$. If no singular direction of $\mu$ at $q$ is
tangent to $\Gamma$, then $\Gamma$ is real analytic near $q$.
\end{lemma}

The slab-area property is a condition on the area densities
$R(u,\cdot)$; the Radon transform is what ties them to the
singularities of $\mu$.
For a smooth compactly supported function $\psi$ on $\R^3$, its
\emph{Radon transform} is
$$
\Radon\psi(u,t)\coloneqq \int_{\phi_u(y)=t}\psi(y),
$$
the integral of $\psi$ over the plane orthogonal to $u$ at
height $t$, for $(u,t)\in\Sph^2\times\R$.
It extends to compactly supported distributions by duality, and
for $\mu$ the coarea formula gives, for $(u,t)\in D$,
$$
\Radon\mu(u,t)=R(u,t).
$$
The wave front set of $\Radon\mu$ lies in the cotangent bundle
$T^*(\Sph^2\times\R)$. There is a natural identification of $T^*_{(u,t)}(\Sph^2\times\R)$ with
$u^\perp\times\R$, and we write its elements as $(u,t;\eta,\tau)$. Let
$
P_uq\coloneqq q-\phi_u(q)u
$
be the orthogonal projection of $q$ onto $u^\perp$. The wave front correspondence for the Radon transform
\cite[Thm.~1.2 \& Lem.~2.5]{chihara-jathar-railo2026}, which goes back to
Boman--Quinto \cite{boman-quinto1987}, gives
\begin{equation}\label{eq:radon-WF}
(q,u)\in\WF_A(\mu)
\quad\Longleftrightarrow\quad
\bigl(u,\phi_u(q);-P_uq,1\bigr)
\in\WF_A(\Radon\mu).
\end{equation}
Thus a singular direction $u$ of $\mu$ at $q$ corresponds to a
singularity of $\Radon\mu$ at the plane through $q$ orthogonal to
$u$.

\begin{lemma}\label{lem:transport}
Let $q\in\Gamma$, $u\in\Sph^2$ be a singular direction of
$\mu$ at $q$ with $u\ne\pm\nu(q)$, and $q'$ be one of the points
$q\pm hu$. If the plane orthogonal to $u$ through $q'$ meets
$\interior(K)$, then $q'\in\Gamma$, and $u$ is a singular direction
of $\mu$ at $q'$.
\end{lemma}

\begin{proof}
Since $u\ne\pm\nu(q)$, the plane orthogonal to $u$ through $q$ is
not a supporting plane and hence meets $\interior(K)$. The plane
through $q'$ meets $\interior(K)$ by hypothesis. Hence
$\phi_u(q)$ and $\phi_u(q')$ both lie in
$(-p(-u),p(u))$. Let $t$ be the larger of the two, so that
$-p(-u)+h<t<p(u)$. By \eqref{eq:periodic},
$
\Radon\mu(u,s)=\Radon\mu(u,s-h)
$
whenever $-p(-u)+h<s<p(u)$.
Since $\Radon\mu$ coincides with its translate
$\Radon\mu(\cdot,\cdot-h)$ on the open set
$\{(u,s):-p(-u)+h<s<p(u)\}$, and wave front sets are invariant
under translation,
$$
(u,t;\eta,\tau)\in\WF_A(\Radon\mu)
\quad\Longleftrightarrow\quad
(u,t-h;\eta,\tau)\in\WF_A(\Radon\mu).
$$
Since $P_uq'=P_uq$ and
$\{t,t-h\}=\{\phi_u(q),\phi_u(q')\}$,
\eqref{eq:radon-WF} shows that $u$ is a singular direction of
$\mu$ at $q$ if and only if it is one at $q'$. Thus $u$ is a
singular direction at $q'$, and so $q'\in\Gamma$ by
Lemma~\ref{lem:WF-basic}.
\end{proof}

\begin{proof}[Proof of Proposition~\ref{prop:C1-analytic}]
Fix $q\in\Gamma$ and put $n\coloneqq \nu(q)$. By Lemma~\ref{lem:boman},
it suffices to show that no singular direction of $\mu$ at $q$ is
tangent to $\Gamma$. Suppose, to the contrary, that $e\in\Sph^2$
is such a direction. Put $w\coloneqq n\times e$, and
$$
u_\theta\coloneqq \cos\theta\,n+\sin\theta\,e,
\qquad
u_\theta'\coloneqq -\sin\theta\,n+\cos\theta\,e,
$$
so that $(u_\theta,u_\theta',w)$ is an orthonormal frame for every
$\theta\in\R$. We proceed in four steps. First, we apply
Lemmas~\ref{lem:watermelon} and~\ref{lem:transport} at $q$ to show
that $\Gamma$ contains the curve $q_\theta\coloneqq q-hu_\theta$ for small
$\theta>0$, with $u_\theta$ singular at $q_\theta$. Second, we apply
Lemma~\ref{lem:watermelon} at $q_\theta$ to obtain singular directions
in $\operatorname{span}\{u_\theta,w\}$ close to $u_\theta$. Third,
we apply Lemma~\ref{lem:transport} uniformly in $\theta$ to the
singular directions obtained in the second step, producing a
two-parameter family of points of $\Gamma$ which contains $q$ when
the second parameter vanishes. Finally, fixing that parameter at a
small nonzero value and differentiating with respect to $\theta$
yields tangent vectors of $\Gamma$ near $q$ which are not orthogonal
to $n$, contradicting the continuity of $\nu$.

\emph{Step 1.}
We first show that $q_\theta\in\Gamma$ for small $\theta>0$. For
$0<\theta<\pi$, $u_\theta$ lies in the open half-plane of
$\operatorname{span}\{n,e\}$ bounded by $\R n$ which contains
$e$, so $u_\theta$ is a singular direction of $\mu$ at $q$ by
Lemma~\ref{lem:watermelon}. The plane orthogonal to $u_\theta$
through $q_\theta$ meets $\interior(K)$ if and only if
$-p(-u_\theta)<\langle q,u_\theta\rangle-h$, since
$\langle q,u_\theta\rangle-h<p(u_\theta)$ holds trivially. For
$\theta=0$ this inequality reads $-p(-n)<p(n)-h$, which holds
because $h<\width(\Gamma)\le \width_n(\Gamma)=p(n)+p(-n)$. Hence, by continuity,
there is $\theta_1\in(0,\pi)$ such that the inequality holds for
$0\le\theta<\theta_1$. Lemma~\ref{lem:transport} now shows that
$q_\theta\in\Gamma$, and that $u_\theta$ is a singular direction of
$\mu$ at $q_\theta$, for $0<\theta<\theta_1$.

\emph{Step 2.}
We next perturb $u_\theta$ through singular directions of $\mu$ at
$q_\theta$ which leave the plane $\operatorname{span}\{n,e\}$. Put
$\nu_\theta\coloneqq \nu(q_\theta)$. The curve $\theta\mapsto q_\theta$ lies
in $\Gamma$ and has velocity $-hu_\theta'$, so
$\langle \nu_\theta,u_\theta'\rangle=0$, that is,
$\nu_\theta\in\operatorname{span}\{u_\theta,w\}$. Moreover
$\nu_\theta\ne\pm u_\theta$, because the plane orthogonal to
$u_\theta$ through $q_\theta$ meets $\interior(K)$ and so is not the
tangent plane of $\Gamma$ at $q_\theta$. Thus $u_\theta$ is a
singular direction of $\mu$ at $q_\theta$ not parallel to
$\nu_\theta$, and
$\operatorname{span}\{\nu_\theta,u_\theta\}=\operatorname{span}\{u_\theta,w\}$.
Let $\alpha(\theta)>0$ be the angle between $u_\theta$ and the line
$\R \nu_\theta$, and for $\varphi\in\R$ put
$
v_{\theta,\varphi}\coloneqq \cos\varphi\,u_\theta+\sin\varphi\,w.
$
For $|\varphi|<\alpha(\theta)$, $v_{\theta,\varphi}$ lies in the
open half-plane of $\operatorname{span}\{u_\theta,w\}$ bounded by
$\R \nu_\theta$ which contains $u_\theta$, and so is a singular
direction of $\mu$ at $q_\theta$ by Lemma~\ref{lem:watermelon}.

\emph{Step 3.}
We now apply Lemma~\ref{lem:transport} to the singular directions
$v_{\theta,\varphi}$  uniformly in $\theta$,
producing a two-parameter family of points of $\Gamma$. Put
$$
F(\theta,\varphi)\coloneqq q_\theta+hv_{\theta,\varphi}
=q+h\bigl((\cos\varphi-1)u_\theta+\sin\varphi\,w\bigr),
$$
so that $F(\theta,0)=q$. Fix a closed interval
$I\subset(0,\theta_1)$. We claim that there is
$\varphi_0\in(0,\pi)$ such that $F(\theta,\varphi)\in\Gamma$ for
$\theta\in I$ and $|\varphi|<\varphi_0$. By
Lemma~\ref{lem:transport}, it suffices that
$|\varphi|<\alpha(\theta)$, so that $v_{\theta,\varphi}$ is a
singular direction of $\mu$ at $q_\theta$ not parallel to
$\nu_\theta$, and that the plane orthogonal to $v_{\theta,\varphi}$
through $F(\theta,\varphi)$ meets $\interior(K)$. Both conditions
hold when $\varphi=0$: the first since $\alpha(\theta)>0$, and the
second since $F(\theta,0)=q$ and $v_{\theta,0}=u_\theta\notin\R n$,
so that the plane in question does not support
$\Gamma$. Both conditions are open in $(\theta,\varphi)$, since
$\alpha$ is continuous and the second condition consists of strict
inequalities between continuous functions. So the claim follows
from the compactness of $I$.

\emph{Step 4.}
For each $0<|\varphi|<\varphi_0$, the curve
$\theta\mapsto F(\theta,\varphi)$, $\theta\in I$, lies in $\Gamma$
and has velocity $h(\cos\varphi-1)u_\theta'\ne0$. Thus $u_\theta'$
is tangent to $\Gamma$ at $F(\theta,\varphi)$, that is,
$\langle\nu(F(\theta,\varphi)),u_\theta'\rangle=0$. As
$\varphi\to0$, $F(\theta,\varphi)\to q$, so continuity of $\nu$
yields $\langle n,u_\theta'\rangle=0$. But
$\langle n,u_\theta'\rangle=-\sin\theta\ne0$ for $\theta\in I$,
which is the desired contradiction.
\end{proof}

\section{Proof of Theorem \ref{thm:smooth}}\label{sec:smooth}

We prove Theorem~\ref{thm:smooth} by showing that the slab-area
property forces convexity, reducing the result to
Theorem~\ref{thm:archimedes}. Throughout this section let $\Gamma$
be a $\C^3$ closed embedded surface satisfying the slab-area
property for some $0<h<\width(\Gamma)$. Let $K$ be the convex hull
of $\Gamma$, $\nu$ be the outward unit normal of $\Gamma$, and
$S=d\nu$ be the shape operator, with principal curvatures
$\kappa_1,\kappa_2$. 
We write
$
H\coloneqq (\kappa_1+\kappa_2)/2
$
and
$
G\coloneqq\kappa_1\kappa_2,
$
for the mean and Gauss curvatures of $\Gamma$ respectively.
The support function $p$, the height functions $\phi_u$, the
level sets $\Gamma_{u,t}$, and the area density $R(u,t)$ are
defined as in Section~\ref{sec:analytic}. Thus $p$ is also the support function
of $K$ and
$
\width(\Gamma)=\min_{u\in\Sph^2}\width_u(\Gamma).
$
Since the derivation of \eqref{eq:periodic} used only the slab-area
property and the coarea formula, it still holds for
$\Gamma$ whenever $t$ and $t+h$ are regular values of $\phi_u$ with
$-p(-u)<t<p(u)-h$.

The proof has two steps. First we use the singularities of the area
density together with its $h$-periodicity to show that
$H\le0$ on the portion of $\Gamma$ lying sufficiently near its
contact set with $\partial K$. We then apply the Hopf boundary
point lemma to show that $\Gamma=\partial K$. This approach requires the following relation between the area density and Gauss curvature. By the standard genericity theorem for height functions, e.g., see
\cite[Lem.~5.3]{dang2016},  there is an open dense set
$\mathcal G\subset\Sph^2$ such that $\phi_u$ is Morse with pairwise
distinct critical values for every $u\in\mathcal G$. The case of extrema in the next lemma is at the heart
of Blaschke's converse to the Archimedes theorem \cite[pp.~83, 161]{blaschke1921}, and
the general case appears in  \cite{vanhove1953}. We include
the short proof for completeness.

\begin{lemma}\label{lem:smooth-critical}
Let $u\in\mathcal G$, let $y$ be a critical point of $\phi_u$, and
put $c\coloneqq \phi_u(y)$. If $y$ is a local minimum or maximum, then
\begin{equation}\label{eq:density-jump}
\lim_{\tau\downarrow0}
\bigl(R(u,c+\tau)-R(u,c-\tau)\bigr)
=
\begin{cases}
\displaystyle \;\;\,\frac{2\pi}{\sqrt{G(y)}},
&\text{if $y$ is a local minimum},\\[6pt]
\displaystyle -\frac{2\pi}{\sqrt{G(y)}},
&\text{if $y$ is a local maximum}.
\end{cases}
\end{equation}
If $y$ is a saddle point, then
\begin{equation}\label{eq:density-log}
\lim_{t\to c}
\frac{R(u,t)}{\log(1/|t-c|)}
=
\frac{2}{\sqrt{|G(y)|}}.
\end{equation}
\end{lemma}

\begin{proof}
Since the critical values are distinct, we may isolate $y$ by a
cutoff equal to $1$ near $y$; the contribution from the rest of the
surface is continuous near $c$. By the Morse lemma there are
coordinates $z=(z_1,z_2)$ centered at $y$ in which
$$
\phi_u
=
c+\frac12\left(\varepsilon_1z_1^2+\varepsilon_2z_2^2\right),
\qquad
\varepsilon_i\in\{1,-1\}.
$$
Writing the area form as
$
dA=\alpha(z)\,dz_1dz_2,
$
we have
$
\alpha(0)=|G(y)|^{-1/2}.
$
At a local extremum, polar coordinates yield
\eqref{eq:density-jump}. At a saddle point, after an orthogonal
change of coordinates we may write $\phi_u=c+z_1z_2$. Since
$\alpha$ is continuous, integration along $z_1z_2=t-c$ gives
$$
R(u,t)
=
\left(\frac{2}{\sqrt{|G(y)|}}+o(1)\right)
\log\frac1{|t-c|},
$$
which proves \eqref{eq:density-log}.
\end{proof}

Let $\mathcal U\subset\Gamma$ be the set of points
$y\in\Gamma\cap\interior(K)$ such that the translate of
$T_y\Gamma$ by $-h\nu(y)$ meets $\interior(K)$. Equivalently,
\begin{equation}\label{eq:U}
\mathcal U
\coloneqq
\left\{
y\in\Gamma:
-p(-\nu(y))
<
\langle y,\nu(y)\rangle-h
<
\langle y,\nu(y)\rangle
<
p(\nu(y))
\right\}.
\end{equation}

\begin{lemma}\label{lem:smooth-H}
$H\le0$ on $\mathcal U$.
\end{lemma}

\begin{proof}
The argument has two parts. First we show that $H\le0$ at
points of $\mathcal U$ where $G\ne0$. Then we rule out $H>0$ at
points where $G=0$.

\emph{Part 1.}
Suppose first that $y\in\mathcal U$, $G(y)\ne0$, and
$u\coloneqq \nu(y)$ belongs to $\mathcal G$. Put $c\coloneqq \phi_u(y)$. By
\eqref{eq:U},
$
-p(-u)<c-h<c<p(u).
$
By Lemma~\ref{lem:smooth-critical}, $R(u,\cdot)$ has at $c$ either
a nonzero jump or a logarithmic singularity. Periodicity forces the
same behavior at $c-h$, so $c-h$ is a critical value. Since the
critical values are distinct, there is a unique critical point $q$
at that height. Equations \eqref{eq:density-jump} and
\eqref{eq:density-log} show that $y$ and $q$ have the same Morse
index and
$
|G(y)|=|G(q)|.
$

Now vary $u$ in a sufficiently small connected neighborhood
$V\subset\mathcal G$. The point $y$ extends to a $\C^1$ critical
branch $y(v)$ with $\nu(y(v))=v$, and periodicity pairs it with a
single $\C^1$ critical branch $q(v)$ satisfying
$
\langle q(v),v\rangle=\langle y(v),v\rangle-h.
$
Differentiating in the direction $\xi\in T_v\Sph^2$, and using that
$dy(v)[\xi]$ and $dq(v)[\xi]$ are orthogonal to $v$, gives
$
\langle y(v)-q(v),\xi\rangle=0.
$
Hence
\begin{equation}\label{eq:parallel-pair}
q(v)=y(v)-hv.
\end{equation}
At $q(v)$ choose the unit normal $v$, whether or not it agrees with
the outward normal there, and let $\widetilde S_q$ be the
corresponding shape operator. Differentiating
\eqref{eq:parallel-pair} gives
$
dq=dy-h\,dv
$
and
$
dv=S_y\,dy=\widetilde S_q\,dq.
$
Thus
$
\widetilde S_q(I-hS_y)=S_y,
$
and, after ordering the principal directions,
$$
\widetilde\kappa_i(q)
=
\frac{\kappa_i(y)}{1-h\kappa_i(y)},
\qquad i=1,2.
$$
Since the Morse indices at $y$ and $q$ agree, $S_y$ and
$\widetilde S_q$ have the same number of positive eigenvalues.
A negative eigenvalue of $S_y$ remains negative under the displayed
transformation. Hence
$
1-h\kappa_i(y)>0
$
for $i=1,2$. Since $|G(y)|=|G(q)|$, we have
$
(1-h\kappa_1(y))(1-h\kappa_2(y))=1.
$
Consequently
$$
2-2hH(y)
=
(1-h\kappa_1(y))+(1-h\kappa_2(y))
\ge2,
$$
so $H(y)\le0$.
For general $y\in\mathcal U$ with $G(y)\ne0$, the Gauss map is a
local diffeomorphism near $y$. Approximate $\nu(y)$ by directions
in $\mathcal G$ and pass to the limit.

\emph{Part 2.}
Suppose towards a contradiction that
$
X\coloneqq \mathcal U\cap\{H>0\}
$
is nonempty. By Part~1, $G=0$ on $X$. Thus the principal curvatures
there are $\kappa=2H>0$ and $0$. By a theorem of Hartman--Nirenberg
\cite[pp.~916--917]{hartman-nirenberg1959}, $X$ is locally ruled by
straight lines, and the tangent plane is constant along each ruling.
Along a unit-speed ruling $\gamma(s)=y_0+se$, the nonzero principal
curvature has the form
\begin{equation}\label{eq:kappa}
\kappa(s)=\frac{\kappa_0}{1+c_0s},
\qquad \kappa_0>0,
\end{equation}
for some constant $c_0$; see \cite[p.~419]{docarmo:csbook}.
Let $\gamma$ be a maximal such ruling in $X$. No endpoint of its
maximal parameter interval can be finite. Indeed, if $s\to\beta$
at a finite endpoint, compactness of $\Gamma$ gives a limit point
$y_*\in\Gamma$. Boundedness of the principal curvatures implies
$1+c_0\beta\ne0$, so $\kappa$ has a positive finite limit and
$H(y_*)>0$. Along the ruling, $\nu$ and
$\langle\gamma,\nu\rangle$ are constant, so \eqref{eq:U} remains
valid at $y_*$. Hence $y_*\in\mathcal U$, and thus $y_*\in X$.
The unique local ruling through $y_*$ then continues $\gamma$ past
$y_*$, contradicting maximality. Hence $\Gamma$ contains an entire
straight line, contrary to compactness.
\end{proof}

\begin{proposition}\label{prop:smooth-convexity}
$\Gamma$ is convex.
\end{proposition}

\begin{proof}
Let
$
C\coloneqq \Gamma\cap\partial K
$
be the contact set. At $y\in C$, a supporting plane of $K$ is
tangent to $\Gamma$, with outward normal $\nu(y)$. Hence
$
\langle y,\nu(y)\rangle=p(\nu(y)).
$
Set
$
d(y)\coloneqq p(\nu(y))-\langle y,\nu(y)\rangle.
$
Then $d=0$ on $C$ and $d>0$ on $\Gamma\setminus C$.
Put
$
\delta\coloneqq \width(\Gamma)-h>0.
$
If $y\in\Gamma\setminus C$ is sufficiently close to $C$, then
$
0<d(y)<\delta/2
$
and
$$
\langle y,\nu(y)\rangle-h+p(-\nu(y))
=
\width_{\nu(y)}(\Gamma)-h-d(y)
\ge\delta-d(y)>0.
$$
Together with $d(y)>0$, this gives \eqref{eq:U}. Hence
$y\in\mathcal U$, and Lemma~\ref{lem:smooth-H} gives $H\le0$
on $\Gamma\setminus C$ sufficiently close to $C$. Suppose that $\Gamma\setminus C$ is nonempty. Then $C$ is a proper
closed subset of $\Gamma$ and has a relative boundary point $y_0$.
After a rigid motion we may assume that $y_0=0$ and $\nu(y_0)=(0,0,1)$,
so that the supporting plane of $K$ at $y_0$ is $\R^2\times\{0\}$
and, near $y_0$, $\Gamma$ is the graph
$
\{(w,f(w)):w\in W\}
$
of a $\C^3$ function $f$ on a neighborhood $W$ of $0$ in $\R^2$. We choose $W$ so small
that $H\le0$ at every noncontact point. Then we have
$$
\operatorname{div}
\frac{\nabla f}{\sqrt{1+|\nabla f|^2}}
=-2H\ge0
$$
on the noncontact region.
Choose a noncontact point $w_1$ sufficiently close to the contact
set, and let $r$ be its distance to that set, with
$
B_r(w_1)\subset W.
$
Then $B_r(w_1)$ contains no contact point and its boundary contains
one, say $w_0$. Write the supporting plane of $K$ at
$(w_0,f(w_0))$ as $z=\ell(w)$. Since $\Gamma$ lies below this plane,
$
v\coloneqq f-\ell
$
satisfies
$
v<0\quad\text{in }B_r(w_1),
$
$
v(w_0)=0,
$
$
\nabla v(w_0)=0.
$
Since $\ell$ is affine, the mean-curvature inequality becomes
$
a^{ij}(\nabla f)v_{ij}\ge0
$
on the disk, where
$$
a^{ij}(\xi)
=
\frac{\delta_{ij}}{\sqrt{1+|\xi|^2}}
-
\frac{\xi_i\xi_j}{(1+|\xi|^2)^{3/2}}.
$$
This is uniformly elliptic on $B_r(w_1)$. The Hopf boundary
point lemma gives a strictly positive outward normal derivative of
$v$ at $w_0$, contradicting $\nabla v(w_0)=0$. Hence
$\Gamma\subset\partial K$.
Finally, $\Gamma$ is a closed surface contained in the topological
sphere $\partial K$, so it is open in $\partial K$ by invariance of
domain, and closed. Hence $\Gamma=\partial K$.
\end{proof}

\begin{note}\label{note:regularity}
The $\C^3$ hypothesis in Theorem~\ref{thm:smooth} is used only in
the proof of Lemma~\ref{lem:smooth-H}, where the Codazzi equations
are applied to establish \eqref{eq:kappa}.  That formula, and hence
Theorem~\ref{thm:smooth}, extends to $\C^2$ surfaces, but we do not
pursue this here.
\end{note}

\begin{note}\label{note:immersed}
Proof of Theorem~\ref{thm:smooth} applies more generally for  immersed surfaces whose self-intersection set is nowhere
dense.
Indeed, the arguments of this section are local and apply with
either choice of a local unit normal $\nu$. In the proof of
Proposition~\ref{prop:smooth-convexity}, $\nu$ is chosen near a
relative boundary point of the contact set to agree with the
outward normal of $K$ there. The hypothesis on the self-intersections
ensures that generic height functions have distinct critical
values, which is used in Lemmas~\ref{lem:smooth-critical}
and~\ref{lem:smooth-H}, since two sheets of $\Gamma$ whose
critical values agree for an open set of directions coincide on an
open set. In the last step, invariance of domain shows that the
immersion is a covering map of $\partial K$, and hence
an embedding. 
\end{note}

\section{Higher Dimensions}
\label{sec:higher-dimensions}
The slab-area property of $\Sph^2$ has the following generalization
to $\Sph^{n-1}$ for $n\geq 3$ \cite{coll-dodd-harrison2017}: for every
$(n-2)$-dimensional linear subspace $E\subset\R^n$, with orthogonal
projection $\pi_E\colon\R^n\to E$, and every measurable set $A$ in
the unit ball of $E$,
\begin{equation}\label{eq:app}
\bigl|\{u\in\Sph^{n-1}:\pi_E(u)\in A\}\bigr|=2\pi|A|.
\end{equation}
Here $|\cdot|$ denotes Hausdorff measure of the
appropriate dimension: $n-1$ for subsets of hypersurfaces, and
$n-2$ for subsets of $E$. We show that
this property for balls $A$ of a single sufficiently small radius,
with the constant of proportionality allowed to depend on $E$,
characterizes $\Sph^{n-1}$. Thus we obtain analogues of Theorems~\ref{thm:archimedes} and \ref{thm:smooth} in higher dimensions.

Let $\Gamma$ be either a
\emph{convex hypersurface}, i.e., the boundary of a convex body
$K\subset\R^n$, or a closed embedded hypersurface, in which
case $K$ denotes the compact region bounded by $\Gamma$. Then  $\inrad(\Gamma)$ is the radius
of a largest ball contained in $K$. Let $\mathcal E\coloneqq\mathbf{Gr}(n-2,n)$ be the Grassmannian of $(n-2)$-dimensional 
subspaces of $\R^n$, with its rotation invariant probability measure
$\sigma$. For $E\in\mathcal E$, $x\in\R^n$, and $r>0$, let
$$
C_r(x,E)\coloneqq\big\{y\in\R^n:|\pi_E(y-x)|\le r\big\}
$$
be the \emph{cylinder} of radius $r$ with the $2$-dimensional axis
$x+E^\perp$. Its projection $\pi_E(C_r(x,E))$ is the ball of radius
$r$ in $E$ centered at $\pi_E(x)$.
We say that $\Gamma$ has the \emph{cylinder-area property} for $r>0$
if, for each $E\in\mathcal E$, all cylinders $C_r(x,E)$ with
$\pi_E(C_r(x,E))\subset\pi_E(K)$ cut off the same area $c_r(E)$ from
$\Gamma$. When $n=3$ and $E=\R u$, we have $C_r(x,E)=S_{2r}(x,u)$, and
$\pi_E(C_r(x,E))\subset\pi_E(K)$ means that both boundary planes of
this slab meet $\Gamma$. So the cylinder-area property for $r$ is the
slab-area property for $h=2r$.

\begin{theorem}\label{thm:hd}
Let $\Gamma\subset\R^n$, be a convex or a $\C^3$
closed embedded hypersurface. If $\Gamma$ has the cylinder-area
property for some $0<r<\inrad(\Gamma)$, then $\Gamma$ is a sphere.
\end{theorem}

The proof follows that of Proposition~\ref{prop:inradius}.
For $x\in K$, we define the \emph{Newtonian potential} of $\Gamma$ by
$$
U(x)\coloneqq\frac{1}{(n-2)|\Sph^{n-1}|}
\int_{y\in\Gamma}\frac{1}{|x-y|^{n-2}},
$$
which agrees with the definition in Section~\ref{sec:potential} when
$n=3$.

\begin{lemma}\label{lem:hd-continuity}
$U$ is continuous on $K$ and harmonic on $\interior(K)$.
\end{lemma}

\begin{proof}
The proof of Lemma~\ref{lem:continuity} applies with the following
modifications. First, $f(t)\coloneqq|\Gamma\cap B_t(x)|\le Ct^{n-1}$
for a constant $C$ independent of $x\in K$. When $\Gamma$ is convex,
this holds with $C=|\Sph^{n-1}|$ by monotonicity of surface area for
nested convex bodies, and when $\Gamma$ is a $\C^3$ closed hypersurface, it follows from
compactness and local graph representations of $\Gamma$.
Then integration by parts gives
$
\int_{y\in\Gamma\cap B_\rho(x)}|x-y|^{2-n}\le(n-1)C\rho,
$
so $U$ is the uniform limit, as $\rho\to0$, of the continuous
functions $U_\rho$ obtained by replacing $|x-y|^{2-n}$ in its
definition with $\min\{\rho^{2-n},|x-y|^{2-n}\}$.
\end{proof}

For $r>0$, set $K_{2r}\coloneqq\{x\in K:\dist(x,\Gamma)>r\}$, in
accordance with Section~\ref{sec:potential}.

\begin{lemma}[The averaging identity]\label{lem:hd-average}
For every $x\in K_{2r}$,
$$
\int_{E\in\mathcal E}\bigl|\Gamma\cap C_r(x,E)\bigr|
=(n-2)|\Sph^{n-1}|\,r^{n-2}\,U(x).
$$
\end{lemma}

\begin{proof}
Fix $x\in K_{2r}$. For $y\in\Gamma$, let
$
Z(y)\coloneqq\{E\in\mathcal E:|\pi_E(y-x)|\le r\}
$
be the set of $E$ with $y\in C_r(x,E)$. We claim that
$$
\sigma\bigl(Z(y)\bigr)=\Bigl(\frac{r}{|x-y|}\Bigr)^{n-2}.
$$
Indeed, fix $E_0\in\mathcal E$. Each $E\in\mathcal E$ is of the form
$gE_0$ for some $g\in O(n)$, and $|\pi_E v|=|\pi_{E_0}g^{-1}v|$ for
$v\in\R^n$. Thus, by invariance of $\sigma$, $\sigma(Z(y))$ is the
normalized area of the set of $u\in\Sph^{n-1}$ with
$|\pi_{E_0}u|\le r/|x-y|$. Since $r/|x-y|<1$, applying
\eqref{eq:app} to the balls of radius $r/|x-y|$ and $1$ in $E_0$
gives the claim.
Now Fubini's theorem yields
\begin{align*}
\int_{E\in\mathcal E}\bigl|\Gamma\cap C_r(x,E)\bigr|
&=\int_{E\in\mathcal E}\int_{y\in\Gamma}\mathbf 1_{Z(y)}(E)
=\int_{y\in\Gamma}\int_{E\in\mathcal E}\mathbf 1_{Z(y)}(E)
=\int_{y\in\Gamma}\sigma\bigl(Z(y)\bigr)\\
&=r^{n-2}\int_{y\in\Gamma}\frac{1}{|x-y|^{n-2}}
=(n-2)|\Sph^{n-1}|\,r^{n-2}\,U(x),
\end{align*}
which completes the proof.
\end{proof}

\begin{proof}[Proof of Theorem~\ref{thm:hd}]
Since $r<\inrad(\Gamma)$, $K_{2r}$ is a nonempty open set. For
$x\in K_{2r}$, the ball $B_r(x)$ lies in $K$, so
$\pi_E(C_r(x,E))=\pi_E(B_r(x))\subset\pi_E(K)$ for every
$E\in\mathcal E$. Thus $|\Gamma\cap C_r(x,E)|=c_r(E)$ by the
cylinder-area property, and Lemma~\ref{lem:hd-average} gives
$$
U(x)=\frac{1}{(n-2)|\Sph^{n-1}|\,r^{n-2}}\int_{E\in\mathcal E}c_r(E).
$$
So $U$ is constant on $K_{2r}$. Since $U$ is harmonic on the connected
set $\interior(K)$, unique continuation shows that $U$ is constant on
$\interior(K)$, and by Lemma~\ref{lem:hd-continuity} on $\Gamma$. If
$\Gamma$ is convex, it is a sphere by Mendez and
Reichel \cite{mendezreichel2000}. If
$\Gamma$ is a closed $\C^3$ hypersurface, the same conclusion follows from 
Reichel \cite[Thm.~2]{reichel1997}.
\end{proof}

\begin{note}
The analogue of the condition $h<\width(\Gamma)$ in
Theorems~\ref{thm:archimedes} and~\ref{thm:smooth} is that
$r<\inrad(\pi_E(K))$ for every $E\in\mathcal E$, which ensures that
the cylinder-area property is not vacuous in any direction. We do not
address this weaker assumption here. The proof of
Theorems~\ref{thm:archimedes} and~\ref{thm:smooth} rests on the
periodicity \eqref{eq:periodic} of the area densities, whose analogue
for $n\ge4$, that the push forward of the area measure of $\Gamma$
under $\pi_E$ assign the same mass to all balls of radius $r$ in
$\pi_E(K)$, is a Pompeiu type condition \cite{zalcman1980} with far less rigid
solutions. The microlocal arguments of Section~\ref{sec:C1} would
also have to be carried out for the transform which integrates over
$2$-planes in $\R^n$, in place of the Radon transform.
\end{note}

\section*{Acknowledgement}
The AI tools Claude (Anthropic) and ChatGPT (OpenAI) were used in the
preparation of this manuscript. The authors have reviewed all AI-assisted
content and take full responsibility for the final manuscript.

\bibliographystyle{amsplain}
\bibliography{references}

\begin{bibdiv}
\begin{biblist}

\bib{archimedes2004}{book}{
      author={Archimedes},
       title={The works of {A}rchimedes. {V}ol. {I}},
   publisher={Cambridge University Press, Cambridge},
        date={2004},
        ISBN={0-521-66160-9},
         url={https://doi.org/10.1017/CBO9780511482557},
        note={The two books on the sphere and the cylinder, Translated into
  English, together with Eutocius' commentaries, with commentary, and critical
  edition of the diagrams by Reviel Netz},
      review={\MR{2093668}},
}

\bib{blaschke1921}{book}{
      author={Blaschke, Wilhelm},
       title={Vorlesungen \"uber {D}ifferentialgeometrie und geometrische
  {G}rundlagen von {E}insteins {R}elativit\"atstheorie. {B}and {I}.
  {E}lementare {D}ifferentialgeometrie},
   publisher={Dover Publications, New York},
        date={1945},
        note={3d ed},
      review={\MR{15247}},
}

\bib{boman2022regularity}{article}{
      author={Boman, Jan},
       title={Regularity of a distribution and of the boundary of its support},
        date={2022},
        ISSN={1050-6926,1559-002X},
     journal={J. Geom. Anal.},
      volume={32},
      number={12},
       pages={Paper No. 300, 17},
         url={https://doi.org/10.1007/s12220-022-01021-y},
      review={\MR{4487764}},
}

\bib{boman-quinto1987}{article}{
      author={Boman, Jan},
      author={Quinto, Eric~Todd},
       title={Support theorems for real-analytic {R}adon transforms},
        date={1987},
        ISSN={0012-7094,1547-7398},
     journal={Duke Math. J.},
      volume={55},
      number={4},
       pages={943\ndash 948},
         url={https://doi.org/10.1215/S0012-7094-87-05547-5},
      review={\MR{916130}},
}

\bib{chihara-jathar-railo2026}{article}{
      author={Chihara, Hiroyuki},
      author={Jathar, Shubham~R.},
      author={Railo, Jesse},
       title={The matrix weighted real-analytic double fibration transforms},
        date={2026},
        ISSN={2737-0690,2737-114X},
     journal={Ann. Fenn. Math.},
      volume={51},
      number={1},
       pages={353\ndash 376},
         url={https://doi.org/10.54330/afm.185114},
      review={\MR{5079078}},
}

\bib{coll-dodd-harrison2017}{article}{
      author={Coll, Vincent},
      author={Dodd, Jeff},
      author={Harrison, Michael},
       title={The {A}rchimedean projection property},
        date={2017},
        ISSN={1615-715X,1615-7168},
     journal={Adv. Geom.},
      volume={17},
      number={3},
       pages={363\ndash 372},
         url={https://doi.org/10.1515/advgeom-2017-0013},
      review={\MR{3680359}},
}

\bib{dang2016}{article}{
      author={Dang, Nguyen~Viet},
       title={The {E}uler characteristic of a surface from its {F}ourier
  analysis in one direction},
        date={2016},
     journal={Math. Res. Lett.},
      volume={23},
      number={5},
       pages={1263\ndash 1279},
         url={https://doi.org/10.4310/MRL.2016.v23.n5.a2},
}

\bib{docarmo:csbook}{book}{
      author={do~Carmo, Manfredo~P.},
       title={Differential geometry of curves and surfaces},
   publisher={Prentice-Hall, Inc., Englewood Cliffs, N.J.},
        date={1976},
        note={Translated from the Portuguese},
      review={\MR{0394451}},
}

\bib{dodd-coll2008}{article}{
      author={Dodd, Jeff},
      author={Coll, Vincent},
       title={Generalizing the equal area zones property of the sphere},
        date={2008},
        ISSN={0047-2468,1420-8997},
     journal={J. Geom.},
      volume={90},
      number={1-2},
       pages={47\ndash 55},
         url={https://doi.org/10.1007/s00022-008-2015-2},
      review={\MR{2465785}},
}

\bib{ghomi2017}{misc}{
      author={Ghomi, Mohammad},
       title={Converse of the {Archimedean} property of the sphere},
        date={2017},
        note={MathOverflow, Question 283109,
  \url{https://mathoverflow.net/questions/283109}},
}

\bib{hartman-nirenberg1959}{article}{
      author={Hartman, Philip},
      author={Nirenberg, Louis},
       title={On spherical image maps whose {J}acobians do not change sign},
        date={1959},
        ISSN={0002-9327},
     journal={Amer. J. Math.},
      volume={81},
       pages={901\ndash 920},
         url={http://dx.doi.org.prx.library.gatech.edu/10.2307/2372995},
      review={\MR{0126812}},
}

\bib{heil-martini1993}{incollection}{
      author={Heil, E.},
      author={Martini, H.},
       title={Special convex bodies},
        date={1993},
   booktitle={Handbook of convex geometry, {V}ol.\ {A}, {B}},
   publisher={North-Holland, Amsterdam},
       pages={347\ndash 385},
      review={\MR{1242985}},
}

\bib{hormander2003analysis}{book}{
      author={H\"ormander, Lars},
       title={The analysis of linear partial differential operators. {I}},
      series={Classics in Mathematics},
   publisher={Springer-Verlag, Berlin},
        date={2003},
        ISBN={3-540-00662-1},
         url={https://doi.org/10.1007/978-3-642-61497-2},
        note={Distribution theory and Fourier analysis, Reprint of the second
  (1990) edition [Springer, Berlin; MR1065993 (91m:35001a)]},
      review={\MR{1996773}},
}

\bib{jeronimo-yee2019}{article}{
      author={Jer\'onimo-Castro, J.},
      author={Yee-Romero, C.},
       title={Characterizations of the {E}uclidean ball by intersections with
  planes and slabs},
        date={2019},
        ISSN={0002-9890,1930-0972},
     journal={Amer. Math. Monthly},
      volume={126},
      number={2},
       pages={151\ndash 157},
         url={https://doi.org/10.1080/00029890.2019.1537759},
      review={\MR{3910629}},
}

\bib{kashiwara1970}{article}{
      author={Kashiwara, Masaki},
       title={On the structure of hyperfunctions (after {M. Sato})},
        date={1970},
     journal={Sugaku no Ayumi},
      volume={15},
       pages={9\ndash 72},
        note={In Japanese},
}

\bib{kim-kim2012}{article}{
      author={Kim, Dong-Soo},
      author={Kim, Young~Ho},
       title={Some characterizations of spheres and elliptic paraboloids},
        date={2012},
        ISSN={0024-3795,1873-1856},
     journal={Linear Algebra Appl.},
      volume={437},
      number={1},
       pages={113\ndash 120},
         url={https://doi.org/10.1016/j.laa.2012.02.013},
      review={\MR{2917435}},
}

\bib{knothe1957}{article}{
      author={Knothe, Herbert},
       title={Inversion of two theorems of {A}rchimedes},
        date={1957},
        ISSN={0026-2285,1945-2365},
     journal={Michigan Math. J.},
      volume={4},
       pages={53\ndash 56},
         url={https://doi.org/10.1307/mmj/1028990176},
      review={\MR{83760}},
}

\bib{landkof1972}{book}{
      author={Landkof, N.~S.},
       title={Foundations of modern potential theory},
      series={Die Grundlehren der mathematischen Wissenschaften},
   publisher={Springer-Verlag, New York-Heidelberg},
        date={1972},
      volume={Band 180},
        note={Translated from the Russian by A. P. Doohovskoy},
      review={\MR{350027}},
}

\bib{mendezreichel2000}{article}{
      author={Mendez, Osvaldo},
      author={Reichel, Wolfgang},
       title={Electrostatic characterization of spheres},
        date={2000},
        ISSN={0933-7741,1435-5337},
     journal={Forum Math.},
      volume={12},
      number={2},
       pages={223\ndash 245},
         url={https://doi.org/10.1515/form.2000.005},
      review={\MR{1740890}},
}

\bib{reichel1997}{article}{
      author={Reichel, Wolfgang},
       title={Radial symmetry for elliptic boundary-value problems on exterior
  domains},
        date={1997},
        ISSN={0003-9527},
     journal={Arch. Rational Mech. Anal.},
      volume={137},
      number={4},
       pages={381\ndash 394},
         url={https://doi.org/10.1007/s002050050034},
      review={\MR{1463801}},
}

\bib{richmond-richmond1993}{article}{
      author={Richmond, B.},
      author={Richmond, T.},
       title={The equal area zones property},
        date={1993},
        ISSN={0002-9890,1930-0972},
     journal={Amer. Math. Monthly},
      volume={100},
      number={5},
       pages={475\ndash 477},
         url={https://doi.org/10.2307/2324302},
      review={\MR{1215535}},
}

\bib{rudin1973}{article}{
      author={Rudin, Walter},
       title={A generalization of a theorem of {A}rchimedes},
        date={1973},
        ISSN={0002-9890,1930-0972},
     journal={Amer. Math. Monthly},
      volume={80},
       pages={794\ndash 796},
         url={https://doi.org/10.2307/2318169},
      review={\MR{328758}},
}

\bib{schneider2014}{book}{
      author={Schneider, Rolf},
       title={Convex bodies: the {B}runn-{M}inkowski theory},
     edition={expanded},
      series={Encyclopedia of Mathematics and its Applications},
   publisher={Cambridge University Press, Cambridge},
        date={2014},
      volume={151},
        ISBN={978-1-107-60101-7},
      review={\MR{3155183}},
}

\bib{sjostrand1982}{incollection}{
      author={Sj\"ostrand, Johannes},
       title={Singularit\'es analytiques microlocales},
        date={1982},
   booktitle={Ast\'erisque, 95},
      series={Ast\'erisque},
      volume={95},
   publisher={Soc. Math. France, Paris},
       pages={1\ndash 166},
      review={\MR{699623}},
}

\bib{sjostrand1994}{article}{
      author={Sj\"ostrand, Johannes},
       title={Remark on extensions of the watermelon theorem},
        date={1994},
        ISSN={1073-2780},
     journal={Math. Res. Lett.},
      volume={1},
      number={3},
       pages={309\ndash 317},
         url={https://doi.org/10.4310/MRL.1994.v1.n3.a3},
      review={\MR{1302646}},
}

\bib{stamm1951}{article}{
      author={Stamm, Otto},
       title={Umkehrung eines {S}atzes von {A}rchimedes \"uber die {K}ugel},
        date={1951},
        ISSN={0025-5858,1865-8784},
     journal={Abh. Math. Sem. Univ. Hamburg},
      volume={17},
       pages={112\ndash 132},
         url={https://doi.org/10.1007/BF02950746},
      review={\MR{41467}},
}

\bib{vanhove1953}{article}{
      author={Van~Hove, L{\'e}on},
       title={The occurrence of singularities in the elastic frequency
  distribution of a crystal},
        date={1953},
     journal={Phys. Rev. (2)},
      volume={89},
       pages={1189\ndash 1193},
}

\bib{yaskin-yaskina2015}{article}{
      author={Yaskin, Vladyslav},
      author={Yaskina, Maryna},
       title={Thick sections of convex bodies},
        date={2015},
        ISSN={0196-8858,1090-2074},
     journal={Adv. in Appl. Math.},
      volume={71},
       pages={174\ndash 189},
         url={https://doi.org/10.1016/j.aam.2015.09.010},
      review={\MR{3406963}},
}

\bib{zalcman1980}{article}{
      author={Zalcman, Lawrence},
       title={Offbeat integral geometry},
        date={1980},
        ISSN={0002-9890,1930-0972},
     journal={Amer. Math. Monthly},
      volume={87},
      number={3},
       pages={161\ndash 175},
         url={https://doi.org/10.2307/2321600},
      review={\MR{562919}},
}

\end{biblist}
\end{bibdiv}

\end{document}